\documentclass{interact}

\usepackage{epstopdf}
\usepackage[caption=false]{subfig}
\usepackage{amsmath}
\usepackage[numbers,sort&compress]{natbib}
\bibpunct[, ]{[}{]}{,}{n}{,}{,}
\makeatletter
\def\NAT@def@citea{\def\@citea{\NAT@separator}}
\makeatother

\usepackage{amsthm}
\usepackage{algorithmic}
\usepackage{amssymb}
\usepackage{mathrsfs}
\usepackage{tikz}
\usepackage{graphicx,epstopdf} 
\usepackage[caption=false]{subfig} 
\usepackage{amsopn}
\usepackage{multicol}

\newcommand{\R}{\mathbb{R}}  
\newcommand{\N}{\mathbb{N}}  

\DeclareMathOperator*{\argmin}{arg\,min}

\DeclareMathOperator{\tr}{tr}

\theoremstyle{plain}
\newtheorem{theorem}{Theorem}[section]
\newtheorem{proposition}[theorem]{Proposition}
\newtheorem{lemma}[theorem]{Lemma}
\newtheorem{corollary}[theorem]{Corollary}

\theoremstyle{definition}
\newtheorem{definition}[theorem]{Definition}
\newtheorem{example}[theorem]{Example}
\theoremstyle{remark}
\newtheorem{remark}[theorem]{Remark}

\begin{document}


\title{Geometric Mean of Partial Positive Definite Matrices with Missing Entries}

\author{
\name{Hayoung Choi\textsuperscript{a}\thanks{(E-mail: hchoi@shanghaitech.edu.cn;skim@chungbuk.ac.kr;shiym@shanghaitech.edu.cn)},
        Sejong Kim\textsuperscript{b,*}, and
        Yuanming Shi\textsuperscript{a}}
\affil{\textsuperscript{a}School of Information Science and Technology, ShanghaiTech University, Shanghai, China;
\textsuperscript{b}Department of Mathematics, Chungbuk National University, Cheongju, Republic of Korea
\textsuperscript{*}Corresponding author}
}

\maketitle

\begin{abstract}
In this paper the geometric mean of partial positive definite matrices with missing entries is considered.
The weighted geometric mean of two sets of positive matrices is defined, and we show whether such a geometric mean holds certain properties which the weighted geometric mean of two positive definite matrices satisfies. Additionally, counterexamples demonstrate that certain properties do not hold.
A Loewner order on partial Hermitian matrices is also defined.
The known results for the maximum determinant positive completion are developed with an integral representation, and the results are applied to
the weighted geometric mean of two partial positive definite matrices with missing entries. Moreover, a relationship between two positive definite completions is established with respect to their determinants, showing  relationship  between  their  entropy  for  a  zero-mean,multivariate Gaussian distribution.
Computational results as well as one application are shown.
\end{abstract}


\section{Introduction}
%
%
%
%

The geometric mean of two positive definite matrices $A$ and $B$ is given by an explicit formula \cite{Ando,PW75}:
\begin{equation}
A \#_{\frac{1}{2}} B = A^{\frac{1}{2}} (A^{-\frac{1}{2}}B A^{-\frac{1}{2}})^{\frac{1}{2}} A^{\frac{1}{2}}.
\end{equation}
This is known as the unique positive definite solution $X$ of the Riccati equation $X A^{-1} X = B$ \cite{KA,LL01}. Moreover, it can be extended to the unique geodesic $t \in [0,1] \mapsto A \#_{t} B = A^{\frac{1}{2}} (A^{-\frac{1}{2}} B A^{-\frac{1}{2}})^{t} A^{\frac{1}{2}}$ connecting from $A$ to $B$ for the Riemannian trace distance on the open convex cone of positive definite matrices \cite{book:Bh,BH06}. This geodesic is called the weighted geometric mean of $A$ and $B$.
A various theories of extending two-variable geometric mean to multi-variable case have been developed: see \cite{ALM,BMP,Han,Kar,LP12}.
A general framework of multivariable operator means containing the multivariable geometric mean as a special case is considered \cite{PALFIA2016951}.
Multi-variable geometric means as well as the two-variable geometric mean of positive definite matrices have been considered as important objects in many pure and applied areas, such as data points in a diverse area of settings \cite{Mo,SH15,AFPA}.

In many research,
there is the potential for missing or incomplete data since data obtained from physical experiments and phenomena are often corrupt or incomplete. The issue with missing data is that nearly all classic and modern statistical and analytical techniques deal with complete data.
It is vital to be able to deal with missing data rather than to delete the incomplete data from the analysis. Over the past twenty years techniques for dealing with missing data in the most appropriate and desirable way possible have been extensively studied in many different fields such as data analysis, statistics, optimization, matrix theory \cite{book:completions,Pigott01,CR09,book:Laurent}.


Covariance matrices are used as features for many signal and image processing applications, including biomedical image segmentation, radar detection, texture analysis, etc.
Recently new geometric approach has been developed for various problems, such as how to measure the distance between two covariance matrices, how to find the average matrix of covariance matrices \cite{HSBB17,AMP17,SBBM17,SHBV17,BR15,Bar13,ABY13,DKZ09,Bar08}.
Especially, in \cite{BB12} Riemannian mean of covariance matrices to space-time adaptive processing is considered.
Recently it becomes more and more important to deal with incomplete covariance matrices in perturbed environment \cite{ZJG16}.
General strategy for completing a partially specified covariance matrix was studied by Dempster \cite{Demp72}.
A zero-mean, multivariate Gaussian distribution on $\R^n$ with density
\begin{equation*}
f(x)=(2\pi)^{-n/2} |\Sigma|^{-1/2} \exp{\bigg\{-\frac{1}{2} x^\top \Sigma^{-1} x\bigg\}}
\end{equation*}
is considered with a partially specified covariance matrix $\Sigma$.
Dempster proposed a completion which maximizes the entropy
\begin{align}
H(f)
&= -\int_{\R^n} \log(f(x))f(x) \mathrm{d}x\\
&= \frac{1}{2}\log(\det{\Sigma})+\frac{1}{2}n(1+ \log{(2\pi))}, \label{e:entropy}
\end{align}
implying that the completion has the maximum determinant. For more information about the maximum determinant and the maximum entropy \cite{FP11,BC82,NN82}.

We consider the geodesic and the geometric mean of two covariance matrices as space-time adaptive processing, additionally with missing entries, i.e., partially specified covariance matrices.
As the process of averaging, the concept of geometric mean of two positive definite matrices with missing entries will play a role to apply a geometric mean to applications in such areas.
In this paper we mainly study the geometric mean of two partial positive definite matrices with missing entries.
After a series of preliminary definitions and known results for a graph, a partial matrix, and the weighted geometric mean of two positive matrices in Section \ref{sec:prelim} that will be used throughout this paper, we consider in Section \ref{sec:mean_subsets} the weighted geometric mean of two subsets of the positive cone. Several meaningful examples for the geometric mean of two subsets are given, and topological properties for it are shown.
Using the geometric mean of two sets of positive definite completions, in Section \ref{sec:mean_partial} we define the geometric mean of two partial positive definite matrices and show that it holds
several of the known properties for the geometric mean of two positive definite matrices. In Section \ref{sec:Loewner}, we define a partial Loewner order for partial Hermitian matrices and characterize the difference of two partial matrices. In Section \ref{sec:maxdet}, the known results for a positive definite completion of maximizing determinant are developed with an integral representation and are applied to the weighted geometric mean of two partial positive definite matrices. Some interesting computational results are found in Section \ref{sec:computation}.

\section{Preliminary}\label{sec:prelim}

\subsection{Graph and Positive Matrix Completion}

In 1981, H. Dym and I. Gohberg studied extensions of band matrices with band inverses \cite{DG81}.
In 1984, R. Grone, C. R. Johnson, E. M. S\'{a}, and H. Wolkowicz considered positive definite completion of partial Hermitian matrices (some entries specified, some missing) \cite{GJSW84}. They showed that if the undirected graph of the specified entries is chordal, a positive definite completion necessarily exists. Johnson, Lundquist, and Naevdal studied positive definite Toeplitz matrix completions in 1997. In \cite{Charles}, they proved that a pattern $P$ of an $(n+1)\times(n+1)$ partial Toeplitz matrix is positive (semi)definite completable if and only if $P=\{k,2k,\ldots,mk\}$
for some $m\in \N$ and $k\in \N$.

Let $V$ be the set of vertices, and let $\{ x, y \}$ denote the edge connecting two points $x, y \in V$.
A \emph{finite undirected graph} is a pair $G=(V,E)$ where the set $V$ of vertices is finite, and the set $E$ of edges is a subset of the set $\{ \{ x, y \}: x, y \in V \}$. In general $E$ may contain \emph{loops} which means that $x=y$. In this paper we assume that the graph always has all loops.
Without loss of generality we assume that $V = \{ 1, 2, \ldots, n \}$.

Define a \emph{$G$-partial matrix} as a set of complex numbers, denoted by
$[a_{ij}]_{G}$ or $A(G)$, where $a_{ij}$ is specified if and only if $\{i,j\}\in E$.
A \emph{completion} of $A(G)=[a_{ij}]_{G}$ is an $n\times n$ matrix $M=[m_{ij}]$ which satisfies $m_{ij}=a_{ij}$ for all $\{i,j\}\in E$. We say that $M$ is a \emph{positive (semi-)definite completion} of $A(G)$ if and only if $M$ is a completion of $A(G)$ and $M$ is positive (semi-)definite.
A \emph{clique} is a subset $C\subset V$ having the property that $\{ x, y \} \in E$ for all $x, y \in C$. A \emph{cycle} in $G$ is a sequence of pairwise distinct vertices $\gamma = (v_1, \ldots, v_s)$ having the property that $\{ v_1,v_2 \}, \{ v_2, v_3 \}, \ldots, \{ v_{s-1}, v_s \}, \{ v_s, v_1 \} \in E$, and $s$ is referred to as the \emph{length} of the cycle. A \emph{chord} of the cycle $\gamma$ is an edge $\{ v_i, v_j \} \in E$ where $1 \leq i < j \leq s$, $\{ v_{i}, v_{j} \} \neq \{ v_{1}, v_{s} \}$, and $|i-j| \geq 2$.

Assume $V = \{ 1, \ldots, n \}$, and let $A(G) = [a_{ij}]_G$ be a $G$-partial matrix. We say that $A(G)$ is a \emph{partial positive (semi-)definite}  if
$$ a_{ji} = \bar{a}_{ij} \quad \text{for all } \{i,j\}\in E$$
and for any clique $C$ of $G$, this principal submatrix $[a_{ij}]_{i,j\in C}$ of $A(G)$ is positive (semi-)definite.
The graph $G$ is called \emph{positive (semi-)definite completable} if any $G$-partial positive (semi-)definite matrix has a positive (semi-)definite completion.

The following proposition shows that the terms \textquotedblleft positive definite completable" and \textquotedblleft positive semi-definite completable" coincide \cite{GJSW84}.
\begin{proposition}\label{prop:positvecompletabl}
A graph $G$ is positive definite completable if and only if $G$ is positive semi-definite completable.
\end{proposition}
From now on, we will henceforth only use the term "completable".
A graph $G$ is \emph{chordal} if there are no minimal cycles of length $\geq 4$. Equivalently, every cycle of length $\geq 4$ has a chord.
This concept characterizes completable graphs \cite{GJSW84}.
\begin{theorem}\label{chordal}
The graph $G$ is completable if and only if $G$ is chordal.
\end{theorem}

\begin{example}\label{ex:non-chordal}
Let $G=(V,E)$ be a graph with $V = \{1,2,3,4\}$ and
$$ E = \{ \{1,1\}, \{1,2\}, \{1,4\}, \{2,2\}, \{2,3\}, \{3,3\}, \{3,4\}, \{4,4\} \}. $$

\begin{multicols}{2}
\begin{center}
\begin{tikzpicture}[node distance=1.8cm, every loop/.style={},
                    thick,main node/.style={circle,draw,font=\sffamily\small\bfseries}]

  \node[main node] (1) {1};
  \node[main node] (2) [below of=1] {2};
  \node[main node] (3) [right of=2] {3};
  \node[main node] (4) [above of=3] {4};

  \path[every node/.style={font=\sffamily\small}]
    (1) edge node [left] {} (4)
    (2) edge node [right] {} (1)
    (3) edge node [right] {} (2)
    (4) edge node [above] {} (3);
\end{tikzpicture}
\end{center}

Example \ref{ex:non-chordal}
: Non-chordal graph.

\columnbreak
\begin{center}
\begin{tikzpicture}[auto, node distance=1.8cm, every loop/.style={},
                    thick,main node/.style={circle,draw,font=\sffamily\small\bfseries}]

  \node[main node] (1) {1};
  \node[main node] (2) [below of=1] {2};
  \node[main node] (3) [right of=2] {3};
  \node[main node] (4) [above of=3] {4};

  \path[every node/.style={font=\sffamily\small}]
    (1) edge node [left] {} (4)
    edge node [left] {} (3)
    (2) edge node [right] {} (1)
    (3) edge node [right] {} (2)
    (4) edge node [above] {} (3);
\end{tikzpicture}
\end{center}

Example \ref{completable_example1}
: A chordal graph.
\end{multicols}

Since the graph $G$ is not chordal, by Theorem \ref{chordal} there exists a partial positive definite matrix $A(G)$ which does not have any positive completions. For example,
the following partial positive definite matrix does not have a positive (semi-)definite completion.
\begin{equation*}
N=
\begin{bmatrix}
1 & -1 & ? & 0 \\
-1 & 2 & 2 & ? \\
? & 2 & 3 & 1 \\
0 & ? & 1 & 1
\end{bmatrix}
\end{equation*}
whose missing entries are denoted by $?$.
\end{example}

\begin{example}\label{completable_example1}
Let $G=(V,E)$ be a graph with $V = \{1,2,3,4\}$ and
$$ E = \{ \{1,1\}, \{1,3\}, \{1,4\}, \{2,2\}, \{2,3\}, \{3,3\}, \{3,4\}, \{4,4\} \}. $$

Since the graph $G$ is chordal, by Theorem \ref{chordal} any matrix $A(G)$ has a positive (semi-)definite completion.
That is, the following partial positive definite matrix has a positive definite completion.
\begin{equation*}
A(G)=
\begin{bmatrix}
* & ? & * & * \\
? & * & * & ? \\
* & * & * & * \\
* & ? & * & *
\end{bmatrix}
\end{equation*}
whose missing entries are denoted by $?$ and specified entries are denoted by $*$.
For example, since the $G$-partial matrices
\begin{equation*}
A(G)=
\begin{bmatrix}
1 & ? & 1 & 1 \\
? & 5 & 1 & ? \\
1 & 1 & 3 & 1 \\
1 & ? & 1 & 2
\end{bmatrix}\quad \text{and   }~
B(G)=
\begin{bmatrix}
4 & ? & 2 & -1 \\
? & 3 & 1 & ? \\
2 & 1 & 6 & 1 \\
-1 & ? & 1 & 3
\end{bmatrix}
\end{equation*}
are partial positive definite,
they have positive definite completions.
\end{example}

Let $M_{m \times n} := M_{m \times n}(\mathbb{C})$ be a set of all $m \times n$ matrices with entries in the field $\mathbb{C}$ of complex numbers. We equip on $M_{m \times n}$ with the inner product defined as
\begin{displaymath}
\langle A, B \rangle := \tr (A^{*} B) = \sum_{i,j=1}^{m,n} \overline{a_{ij}} b_{ij},
\end{displaymath}
for $A = [a_{ij}], B = [b_{ij}] \in M_{m \times n}$, where $A^{*} = \bar{A}^{T}$ is a complex conjugate transpose of $A$. The inner product naturally gives us an $l_{2}$ norm, known as the Frobenius norm and Hilbert-Schmidt norm, defined by
\begin{displaymath}
\Vert A \Vert_{2} = [ \tr (A^{*} A) ]^{1/2}.
\end{displaymath}
We simply denote as $M_{n} := M_{n \times n}$. We also denote as $\textrm{GL}_{n}$ the general linear group in $M_{n}$.

\begin{remark} \label{R:compatible}
The operator norm of $A \in M_{n}$ is defined as
\begin{displaymath}
\Vert A \Vert := \underset{\Vert x \Vert_2 = 1}{\max} \Vert A x \Vert_{2}.
\end{displaymath}
Note that
\begin{displaymath}
\Vert A \Vert_{2} = \left[ \sum_{i=1}^{n} \sigma_{i}^{2}(A) \right]^{1/2} \ \textrm{and} \ \Vert A \Vert = \sigma_{1}(A),
\end{displaymath}
where $\sigma_{1}(A) \geq \cdots \geq \sigma_{n}(A)$ are (non-negative) singular values of $A$ in decreasing order. Since $\Vert A \Vert \leq \Vert A \Vert_{2} \leq n \Vert A \Vert_{2}$, two norms $\Vert \cdot \Vert_{2}$ and $\Vert \cdot \Vert$ are compatible.
\end{remark}

Let $\mathbb{H} \subset M_{n}$ be the real vector space of all Hermitian matrices, and let $\mathbb{P} \subset \mathbb{H}$ be the open convex cone of $n \times n$ positive definite matrices. Then the closure $\overline{\mathbb{P}}$ of $\mathbb{P}$ consists of all $n \times n$ positive semi-definite matrices. For any $A, B \in \mathbb{H}$ we denote as $A \leq B$ if and only if $B - A \in \overline{\mathbb{P}}$, and $A < B$ if and only if $B - A \in \mathbb{P}$. This is known as the Loewner partial ordering \cite[Section 7.7]{HJ}.

The Frobenius norm $\Vert \cdot \Vert_{2}$ gives rise to the Riemannian trace metric on $\mathbb{P}$ given by
\begin{equation}\label{eq:distancePD}
\delta(A, B) = \Vert \log (A^{-1/2} B A^{-1/2}) \Vert_{2}
\end{equation}
for any $A, B \in \mathbb{P}$. Then $\mathbb{P}$ is a Cartan-Hadamard manifold, a simply connected complete Riemannian manifold with non-positive sectional curvature. The curve
\begin{displaymath}
[0,1] \ni t \mapsto A \#_{t} B := A^{1/2} (A^{-1/2} B A^{-1/2})^{t} A^{1/2}
\end{displaymath}
is the unique geodesic from $A$ to $B$, called the \emph{weighted geometric mean} of positive definite matrices $A$ and $B$.
Note that $A \# B := A \#_{1/2} B$ is the unique midpoint between $A$ and $B$ for the Riemannian metric. We review several known properties of the weighted geometric mean on the open convex cone $\mathbb{P}$ of positive definite matrices.

\begin{theorem} \label{T:geomean}
The weighted two-variable geometric mean satisfies the following: for any $A$, $B$, $C$, $D \in \mathbb{P}$ and $t \in [0,1]$
\begin{itemize}
\item[(1)] $A \#_{t} B = A^{1-t} B^{t}$ if $A$ and $B$ commute.
\item[(2)] $(a A) \#_{t} (b B) = a^{1-t} b^{t} (A \#_{t} B)$ for any $a, b >0$.
\item[(3)] $A \#_{t} B = B \#_{1-t} A$.
\item[(4)] $A \#_{t} B \leq C \#_{t} D$ whenever $A \leq C$ and $B \leq D$.
\item[(5)] $[0,1] \times \mathbb{P}  \times \mathbb{P} \ni (t, A, B) \mapsto A \#_{t} B \in \mathbb{P}$ is continuous.
\item[(6)] $S^{*} (A \#_{t} B) S = (S^{*} A S) \#_{t} (S^{*} B S)$ for any invertible $S \in \mathrm{GL}_{n}$.
\item[(7)] $[(1 - \lambda) A + \lambda B] \#_{t} [(1 - \lambda) C + \lambda D] \geq (1 - \lambda) (A \#_{t} C) + \lambda (B \#_{t} D)$ for any $\lambda \in [0,1]$.
\item[(8)] $(A \#_{t} B)^{-1} = A^{-1} \#_{t} B^{-1}$.
\item[(9)] $\det (A \#_{t} B) = (\det A)^{1-t} (\det B)^{t}$.
\item[(10)] $[(1-t) A^{-1} + t B^{-1}]^{-1} \leq A \#_{t} B \leq (1-t) A + t B$ for any $t \in [0,1]$.
\end{itemize}
\end{theorem}

\begin{remark} \label{R:Riemann}
Item (5) can be described as the map $[0,1] \times \mathbb{P}  \times \mathbb{P} \ni (t, A, B) \mapsto A \#_{t} B \in \mathbb{P}$ is continuous with respect to the Riemannian trace metric $\delta$:
\begin{align*}
\delta(A \#_{s} B, C \#_{t} D)
& \leq \delta(A \#_{s} B, A \#_{t} B) + \delta(A \#_{t} B, C \#_{t} D)\\
& \leq |s-t| \delta(A, B) +  (1-t) \delta(A, C) + t \delta(B, D)
\end{align*}
for any $A, B, C, D \in \mathbb{P}$ and $s,t \in [0,1]$.
\end{remark}

\begin{remark} \label{R:PSD}
One can define the weighted geometric mean for positive semi-definite matrices $A$ and $B$ such as
\begin{equation} \label{E:semi-geomean}
A \#_{t} B := \lim_{\epsilon \to 0^{+}} (A + \epsilon I) \#_{t} (B + \epsilon I).
\end{equation}
By Theorem \ref{T:geomean} (4), $(A + \epsilon I) \#_{t} (B + \epsilon I)$ is monotone decreasing on $\epsilon > 0$ and is bounded below by $O$. So it converges, and thus, the equation \eqref{E:semi-geomean} is well-defined.
\end{remark}

\section{Weighted geometric mean of two subsets of the positive cone}\label{sec:mean_subsets}


In this paper we deal with the following weighted geometric mean of two subsets of $\mathbb{P}$ and see its geometric properties.
\begin{definition} \label{D:geomean}
Let $\mathcal{S} \subset \mathbb{P}$ and $\mathcal{T} \subset \mathbb{P}$, and let $t \in [0,1]$. The weighted geometric mean of two subsets of positive definite matrices is defined by
$$ \mathcal{S} \#_{t} \mathcal{T} := \{ S \#_{t} T ~|~ S \in \mathcal{S},~ T \in \mathcal{T}\}. $$
\end{definition}

\begin{example} \label{E:geomean-ST}
The weighted geometric mean of two subsets of $\mathbb{P}$ has a very important concept in the theory of operator, matrix means, and approximation. In order to see this insight, we give several examples in the following.
\begin{itemize}
\item[(1)] The weighted geometric mean $A \#_{t} B$ of $A$ and $B$ in $\mathbb{P}$ is a special example of that of two subsets $\mathcal{S} = \{ A \} \subset \mathbb{P}$ and $\mathcal{T} = \{ B \} \subset \mathbb{P}$. Moreover, if two subsets $\mathcal{S}$ and $\mathcal{T}$ have cardinalities of $p$ and $q$, respectively, then the cardinality of $\mathcal{S} \#_{t} \mathcal{T}$ is less than or equal to $p q$.

\item[(2)] For given $A, B \in \overline{\mathbb{P}}$, consider $\mathcal{S} = \{ A + \epsilon I: \epsilon_1 > 0 \}$ and $\mathcal{T} = \{ B + \epsilon_2 I: \epsilon > 0 \}$. Then $\mathcal{S}, \mathcal{T} \subset \mathbb{P}$, and
\begin{displaymath}
(A + \epsilon_1 I) \#_{t} (B + \epsilon_2 I) \in \mathcal{S} \#_{t} \mathcal{T},
\end{displaymath}
which is a generalized form of the right-hand side in the limit of \eqref{E:semi-geomean}.

\item[(3)] For $A, B, C, D \in \mathbb{P}$ let $\mathcal{S} = [A, B] = \{ X \in \mathbb{P}: A \leq X \leq B \}$ and $\mathcal{T} = [C, D] = \{ Y \in \mathbb{P}: C \leq Y \leq D \}$. Then by monotonicity of the geometric mean in Theorem \ref{T:geomean} (4),
$$ \mathcal{S} \#_{t} \mathcal{T} \subseteq [A \#_{t} C, B \#_{t} D]. $$

\item[(4)] For $A, B \in \mathbb{P}$ let $\mathcal{S} = \{ X \in \mathbb{H}: \| A - X \| \leq r_1 \}$ and $\mathcal{T} = \{ Y \in \mathbb{H}: \| B-Y \| \leq r_2 \}$. Then $\mathcal{S},\mathcal{T} \subset \mathbb{P}$ for sufficiently small $r_{1}, r_{2} > 0$. So the weighted geometric mean of $\mathcal{S}$ and $\mathcal{T}$, $\mathcal{S} \#_t \mathcal{T}$, can be considered as a set of approximations of $A \#_t B$.

\end{itemize}
\end{example}

Especially, for the case of (3) in Example \ref{E:geomean-ST} the following property which is similar to \cite[Theorem 3.4]{LL01} holds.
\begin{proposition}
Assume that two subsets $\mathcal{S}$ and $\mathcal{T}$ of $\mathbb{P}$ are totally ordered with respect to Loewner order. Then
\begin{equation} \label{E:max}
M \# N = \max{ \Bigg\{
X \in \mathbb{H} :
\left( \begin{array}{cc}
A & X \\
X & B \\
\end{array}
\right)
\geq 0, A \in \mathcal{S}, B \in \mathcal{T}
\Bigg\}},
\end{equation}
where $M$ is the maximum element of $\mathcal{S}$ and $N$ is the maximum element of $\mathcal{T}$.
\end{proposition}
\begin{proof}
It is known from Theorem 4.1.3 (iii) in \cite{book:Bh} that for given $A \in \mathcal{S}, B \in \mathcal{T}$
\begin{equation*}
A \# B = \max{ \Bigg\{
X \in \mathbb{H} :
\left( \begin{array}{cc}
A & X \\
X & B \\
\end{array}
\right)
\geq 0
\Bigg\}}.
\end{equation*}
Since two subsets $\mathcal{S}$ and $\mathcal{T}$ of $\mathbb{P}$ are totally ordered with respect to Loewner order, $A \leq M$ for all $A \in \mathcal{S}$ and $B \leq N$ for all $B \in \mathcal{T}$. By the monotonicity of geometric mean in Theorem \ref{T:geomean} (4), $A \# B \leq M \# N$, and hence, we obtain \eqref{E:max}.
\end{proof}

\begin{theorem} \label{T:geomean-subsets}
Let $\mathcal{S}, \mathcal{T} \subset \mathbb{P}$, and let $t \in [0,1]$. Then
\begin{itemize}
\item[(1)] if $\mathcal{S}$ and $\mathcal{T}$ are bounded, then so is $\mathcal{S} \#_{t} \mathcal{T}$,
\item[(2)] if $\mathcal{S}$ and $\mathcal{T}$ are closed, then so is $\mathcal{S} \#_{t} \mathcal{T}$.
\end{itemize}
Hence, $\mathcal{S} \#_{t} \mathcal{T}$ is compact whenever $\mathcal{S}, \mathcal{T}$ are compact.
\end{theorem}

\begin{proof}
Note that it is enough to show (1) and (2) for compactness, since $\mathbb{P}$ is a subset of Euclidean space $\mathbb{H}$.
\begin{itemize}
\item[(1)] Assume that $\mathcal{S}$ and $\mathcal{T}$ are bounded. By Remark \ref{R:compatible} $\Vert A \Vert \leq c$ for all $A \in \mathcal{S}$ and some constant $c > 0$, and $\Vert B \Vert \leq d$ for all $B \in \mathcal{T}$ and some constant $d > 0$. Then $0 < A \leq c I$ for all $A \in \mathcal{S}$, and $0 < B \leq d I$ for all $B \in \mathcal{T}$. By Example \ref{E:geomean-ST} (3), and Theorem \ref{T:geomean} (1) and (4), it follows that
\begin{displaymath}
0 < A \#_{t} B \leq c^{1-t} d^{t} I.
\end{displaymath}
That is, $\Vert A \#_{t} B \Vert \leq c^{1-t} d^{t}$, and thus, $\mathcal{S} \#_{t} \mathcal{T}$ is bounded.

\item[(2)] Let $\mathcal{S}$ and $\mathcal{T}$ be closed. Assume that sequences $A_{n} \in \mathcal{S}$ and $B_{n} \in \mathcal{S}$ converge to $A \in \mathcal{S}$ and $B \in \mathcal{T}$ with respect to the Riemannian distance, respectively. By continuity of the geometric mean in Theorem \ref{T:geomean} (5) and Remark \ref{R:Riemann},
\begin{displaymath}
A_{n} \#_{t} B_{n} \to A \#_{t} B \in \mathcal{S} \#_{t} \mathcal{T}.
\end{displaymath}
Thus, $\mathcal{S} \#_{t} \mathcal{T}$ is closed.
\end{itemize}
\end{proof}

\begin{remark}
Note that the union of a finite number of compact subsets of $\mathbb{P}$ is compact and the intersection of any family of compact subspace of $\mathbb{P}$ is compact. If
$\{\mathcal{S}_i\}$ and $\{\mathcal{T}_j\}$ are collections of compact subsets of $\mathbb{P}$, then
$$\bigg( \bigcup_{k=1}^n \mathcal{S}_{i_k} \bigg)\# \bigg( \bigcup_{k=1}^m \mathcal{T}_{j_k}\bigg) \quad \text{and} \quad
\bigg( \bigcap_i \mathcal{S}_i \bigg)\# \bigg( \bigcap_j \mathcal{T}_j\bigg)$$  are compact.
\end{remark}

\begin{remark} \label{R:convex}
Assume that $\mathcal{S}, \mathcal{T} \subset \mathbb{P}$ are convex. Let $A, B \in \mathcal{S}$ and $C, D \in \mathcal{T}$. Since $\mathcal{S}$ and $\mathcal{T}$ are convex, $(1 - \lambda) A + \lambda B \in \mathcal{S}$ and $(1 - \lambda) C + \lambda D \in \mathcal{T}$ for any $\lambda \in [0,1]$, and hence, $[(1 - \lambda) A + \lambda B] \#_{t} [(1 - \lambda) C + \lambda D] \in \mathcal{S} \#_{t} \mathcal{T}$. On the other hand, it holds from the joint concavity of geometric mean in Theorem \ref{T:geomean} (7) that
$$ [(1 - \lambda) A + \lambda B] \#_{t} [(1 - \lambda) C + \lambda D] \geq (1 - \lambda) (A \#_{t} C) + \lambda (B \#_{t} D) $$
for any $\lambda \in [0,1]$. It is questionable whether or not $(1 - \lambda) (A \#_{t} C) + \lambda (B \#_{t} D)\in \mathcal{S} \#_{t} \mathcal{T}$. If it is true, then we can say that $\mathcal{S} \#_{t} \mathcal{T}$ is convex.
\end{remark}

\section{Geometric mean of partial positive matrices}\label{sec:mean_partial}

From now on, we consider the geometric mean of partial positive matrices. In this paper, we assume that
any graph $G$ always includes all loops.
That is, any partial matrix does not have missing entries on diagonal.
Recall that $\mathbb{H} \subset M_{n}$ is the real vector space of all Hermitian matrices, $\mathbb{P} \subset \mathbb{H}$ is the open convex cone of $n \times n$ positive definite matrices, and the closure $\bar{\mathbb{P}}$ of $\mathbb{P}$ consists of all $n \times n$ positive semi-definite matrices.
We define
\begin{align*}
\mathbb{H}(G) :&= \{ A(G) : A(G) \text{ is a $n \times n$ $G$-partial Hermitian matrix.} \}, \\
\mathbb{P}(G) :&= \{ A(G) \in \mathbb{H}(G) : A(G) \text{ is a partial positive definite matrix.} \}.
\end{align*}

For a given $G$-partial matrix $A(G)$, we denote as
$\mathfrak{p}[A(G)]$ and $\mathfrak{p}^+[A(G)]$ the sets of all positive semi-definite and positive definite completions of $A(G)$, respectively.

\begin{theorem} \label{T:property}
Let $A(G)$ be a partial positive semidefinite matrix with a completable graph $G$.
Then $\mathfrak{p}[A(G)]$ is nonempty, convex, and compact.
\end{theorem}

\begin{proof}
Since $G$ is a positive completable graph and $A(G)$ is a partial positive semi-definite matrix, clearly $\mathfrak{p}[A(G)]$ is nonempty.
If $M, N \in \mathfrak{p}[A(G)]$, then
$(1-t) M + t N \in \mathfrak{p}[A(G)]$ for $t\in [0,1]$, so $\mathfrak{p}[A(G)]$ is convex. Since we assume that diagonal entries are given, $\mathfrak{p}[A(G)]$ is bounded, since
\begin{displaymath}
\Vert M \Vert_{2} = \left[ \sum_{i=1}^{n} \lambda_{i}^{2} (M) \right]^{1/2} \leq \tr (M) = \sum_{i=1}^{n} m_{ii} < \infty
\end{displaymath}
for any $M \in \mathfrak{p}[A(G)]$, where $\lambda_{i}(M)$ denotes the non-negative eigenvalue of $M$.

Now we show that $\mathfrak{p}[A(G)]$ is closed.
Let $M_{k} = [m_{ij}^{(k)}]$ be a sequence in $\mathfrak{p}[A(G)]$ converging to $M = [m_{ij}]$ in the Frobenius norm. Since $M_{k} \in \overline{\mathbb{P}}$ for all $k$, we have $M \in \overline{\mathbb{P}}$. Moreover, since
\begin{displaymath}
| m_{ij}^{(k)} - m_{ij} | \leq \Vert M_{k} - M \Vert_{2},
\end{displaymath}
we have that $m_{ij}^{(k)} \to m_{ij}$ as $k \to \infty$ for all $1 \leq i, j \leq n$. Since $m_{ij}^{(k)} = m_{ij}$ for all $k$ and $\{ i,j \} \in E$, taking the limit as $k \to \infty$ yields that $m_{ij} = a_{ij}$ for all $\{ i,j \} \in E$. So, $M$ is a positive semi-definite completion of $A(G)$, that is, $M \in \mathfrak{p}[A(G)]$.
\end{proof}

\begin{remark} \label{R:property}
For a  partial positive definite matrix $A(G)$ with a completable graph $G$, one can see easily that $\mathfrak{p}^{+}[A(G)]\in \mathbb{P}(G)$ is nonempty, convex, and bounded.
Since $\mathfrak{p}^{+}[A(G)] \subset \mathfrak{p}[A(G)]$, we have $\overline{\mathfrak{p}^{+}[A(G)]} \subset \overline{\mathfrak{p}[A(G)]} = \mathfrak{p}[A(G)]$ by Theorem \ref{T:property}, where $\overline{\mathfrak{p}^{+}[A(G)]}$ is the closure of $\mathfrak{p}^{+}[A(G)]$. On the other hand, it is questionable that  $\mathfrak{p}[A(G)] \subset \overline{\mathfrak{p}^{+}[A(G)]}$.
\end{remark}

Let $G$ and $F$ be given completable graphs.
One can naturally ask to define the geometric mean of partial positive definite matrices $A(G)$ and $B(F)$. Using the geometric mean of subsets of $\mathbb{P}$ in Definition \ref{D:geomean}, we define the geometric mean of two partial positive definite matrices $A(G)$ and $B(F)$ as
\begin{equation}
A(G) \#_{t} B(F) := \mathfrak{p}^{+}[A(G)] \#_{t} \mathfrak{p}^{+}[B(F)],
\end{equation}
where $t \in [0,1]$.

\begin{remark}
Using Remark \ref{R:PSD}, one can define the geometric mean of partial positive semi-definite matrices $A(G)$ and $B(F)$ as
\begin{displaymath}
\mathfrak{p}[A(G)] \#_{t} \mathfrak{p}[B(F)].
\end{displaymath}
There are some results for the geometric mean of positive semi-definite matrices \cite{KA}, but it holds more limited properties than that of positive definite matrices.
So we consider in this article the geometric mean of partial positive definite matrices.
\end{remark}

\begin{remark} \label{R:bounded}
By Theorem \ref{T:geomean-subsets} and Reamrk \ref{R:property},  $A(G) \#_{t} B(F)$ is bounded for partial positive definite matrices $A(G)$ and $B(F)$ with completable graphs $G$ and $F$.
\end{remark}

It would be interesting to find some properties for $A(G) \#_{t} B(F)$ corresponding to those in Theorem \ref{T:geomean}.
\begin{remark}
Note that
$A(G) \#_{t} A(G) \neq \mathfrak{p}^{+}[A(G)]$, since $A_{1} \#_{t} A_{2}$ may not be a positive definite completion of $A(G)$ even though $A_{1}, A_{2} \in \mathfrak{p}^{+}[A(G)]$. For instance, see Example \ref{completable_example1}. Let
\begin{displaymath}
A_{1} =
\begin{bmatrix}
1 & 1 & 1 & 1 \\
1 & 5 & 1 & 1 \\
1 & 1 & 3 & 1 \\
1 & 1 & 1 & 2
\end{bmatrix}, \ \
A_{2} =
\begin{bmatrix}
1 & -1 & 1 & 1 \\
-1 & 5 & 1 & -1 \\
1 & 1 & 3 & 1 \\
1 & -1 & 1 & 2
\end{bmatrix}.
\end{displaymath}
Then $A_{1}$ and $A_{2}$ are positive definite completions of $A(G)$. However,
\begin{displaymath}
A_{1} \# A_{2} \approx
\begin{bmatrix}
0.8750 & -0.0769 & 1 & 0.8750 \\
-0.0769 & 4.1251 & 1 & -0.0769 \\
1 & 1 & 3 & 1 \\
0.8750 & -0.0769 & 1 & 1.8750
\end{bmatrix},
\end{displaymath}
which is not a positive definite completion of $A(G)$. Clearly, it holds that  $\mathfrak{p}^{+}[A(G)] \subset A(G) \# A(G)  $.
\end{remark}

For $A(G), B(G) \in \mathbb{H}(G)$ with a given graph $G$, \emph{the sum of two $G$-partial matrices} and \emph{the scalar product of a $G$-partial matrix}, denoted by $A(G)+B(G)$ and $\alpha A(G)$ are defined by
\begin{align*}
[A(G)+B(G)]_{ij} &=
\begin{cases}
a_{ij}+b_{ij} & \text{if } \{ i,j \} \in E,\\
\text{missing} & \text{otherwise,}
\end{cases}\\
[\alpha A(G)]_{ij} &=
\begin{cases}
\alpha a_{ij} & \text{if } \{ i,j \} \in E,\\
\text{missing} & \text{otherwise,}
\end{cases}
\end{align*}
where
$A(G) = [a_{ij}]_{\{i,j\} \in E}$ and $B(G) = [b_{ij}]_{\{i,j\} \in E}$, respectively.

Note that $A(G)+B(G)$ and $\alpha A(G)$ are  in $\mathbb{H}(G)$.
It is natural to define \emph{the difference of two $G$-partial matrices} as
$A(G)-B(G):=A(G)+(-1)B(G)$.

Let $\mathcal{S} \subset \mathbb{P}$ and let $\alpha > 0$. For convenience, we denote
\begin{equation*}
\begin{split}
\alpha \mathcal{S} : & = \{ \alpha A: A \in \mathcal{S} \}, \\
\mathcal{S}^{-1} : & = \{ A^{-1}: A \in \mathcal{S} \}.
\end{split}
\end{equation*}
Note that $\alpha \mathcal{S}, \mathcal{S}^{-1} \subset \mathbb{P}$.
It is trivial that
$$ \alpha \mathfrak{p}^{+}[A(G)] = \mathfrak{p}^{+}[\alpha A(G)], \quad
\mathfrak{p}^{+}[A(G) + B(G)] = \mathfrak{p}^{+}[A(G)] + \mathfrak{p}^{+}[B(G)]. $$

\begin{proposition}
For completable graphs $G$ and $F$, let $A(G)$ and $B(F)$ be partial positive definite matrices. Then the following hold:
\begin{itemize}
\item[(1)] $a^{1-t} b^{t} \left( \mathfrak{p}^{+}[A(G)] \#_{t} \mathfrak{p}^{+}[B(F)] \right) = (a \mathfrak{p}^{+}[A(G)]) \#_{t} (b \mathfrak{p}^{+}[B(F)])$ for any $a, b > 0$,
\item[(2)] $\mathfrak{p}^{+}[A(G)] \#_{t} \mathfrak{p}^{+}[B(F)] = \mathfrak{p}^{+}[B(F)] \#_{1-t} \mathfrak{p}^{+}[A(G)]$, and
\item[(3)] $\left( \mathfrak{p}^{+}[A(G)] \#_{t} \mathfrak{p}^{+}[B(F)] \right)^{-1} = \mathfrak{p}^{+}[A(G)]^{-1} \#_{t} \mathfrak{p}^{+}[B(F)]^{-1}$.
\end{itemize}
\end{proposition}

\begin{proof}
Let $A(G) \in \mathbb{P}(G)$ and $B(F) \in \mathbb{P}(F)$.
\begin{itemize}
\item[(1)] Let $A \#_{t} B \in \mathfrak{p}^{+}[A(G)] \#_{t} \mathfrak{p}^{+}[B(F)]$, where $A \in \mathfrak{p}^{+}[A(G)]$ and $B \in \mathfrak{p}^{+}[B(F)]$. Then $a A \in a \mathfrak{p}^{+}[A(G)] = \mathfrak{p}^{+}[a A(G)]$, and $b B \in \mathfrak{p}^{+}[b B(G)]$. By Theorem \ref{T:geomean} (2),
\begin{displaymath}
a^{1-t} b^{t} (A \#_{t} B) = (a A) \#_{t} (b B) \in (a \mathfrak{p}^{+}[A(G)]) \#_{t} (b \mathfrak{p}^{+}[B(F)]).
\end{displaymath}
So $a^{1-t} b^{t} \left( \mathfrak{p}^{+}[A(G)] \#_{t} \mathfrak{p}^{+}[B(F)] \right) \subseteq (a \mathfrak{p}^{+}[A(G)]) \#_{t} (b \mathfrak{p}^{+}[B(F)])$.

Let $C \#_{t} D \in (a \mathfrak{p}^{+}[A(G)]) \#_{t} (b \mathfrak{p}^{+}[B(F)])$, where $C \in \mathfrak{p}^{+}[a A(G)]$ and $D \in \mathfrak{p}^{+}[b B(F)]$. Then $a^{-1} C \in \mathfrak{p}^{+}[A(G)]$ and $b^{-1} D \in \mathfrak{p}^{+}[B(F)]$, and furthermore, by Theorem \ref{T:geomean} (2)
\begin{displaymath}
C \#_{t} D = a^{1-t} b^{t} [ (a^{-1} C) \#_{t} (b^{-1} D) ] \in a^{1-t} b^{t} \left( \mathfrak{p}^{+}[A(G)] \#_{t} \mathfrak{p}^{+}[B(F)] \right).
\end{displaymath}

\item[(2)] By Theorem \ref{T:geomean} (3), $A \#_{t} B = B \#_{1-t} A$ for any $A \in \mathfrak{p}^{+}[A(G)]$ and $B \in \mathfrak{p}^{+}[B(F)]$. Thus, it is proved.

\item[(3)] By Theorem \ref{T:geomean} (8), $(A \#_{t} B)^{-1} = A^{-1} \#_{t} B^{-1}$ for any $A \in \mathfrak{p}^{+}[A(G)]$ and $B \in \mathfrak{p}^{+}[B(F)]$. Thus, it is proved.
\end{itemize}
\end{proof}

\section{Difference of partial matrices and Loewner order on partial matrices}\label{sec:Loewner}

Analogous to the Loewner order on $\mathbb{H}$, we define the relation for $G$-partial semi-definite matrices.
\begin{definition}\label{D:order}
For a given completable graph $G$, we define the relation $\leq$ on $\mathbb{H}(G)$ as follows:
\begin{itemize}
\item[(i)] $A(G) \geq B(G)$ if and only if $A(G)-B(G) \in \overline{\mathbb{P}}(G)$;
\item[(ii)] $A(G) > B(G)$ if and only if $A(G)-B(G)\in \mathbb{P}(G)$.
\end{itemize}
\end{definition}

\begin{theorem}
The relation $\leq$ is indeed a partial order on $\mathbb{H}(G)$ with a completable graph $G$.
\end{theorem}

\begin{proof}
Let $G$ be a completable graph.

\noindent ({\bf Reflexive}) Since $A(G)-A(G)$ is the zero matrix, it holds $A(G) \leq A(G)$ for all $A(G)\in \mathbb{H}(G)$.

\noindent ({\bf Anti-symmetric}) Suppose that $A(G) \leq B(G)$ and $B(G) \leq A(G)$ for $A(G), B(G) \in \mathbb{H}(G)$. Since $A(G)-B(G)$ and $A(G)-B(G)$ partial positive semidefinite, then the diagonal entries of $A(G)-B(G)$ must be 0, implying $A(G)-B(G)=0$. Thus $A(G)=B(G)$.

\noindent ({\bf Transitive}) Suppose that $A(G) \leq B(G)$ and $B(G) \leq C(G)$ for $A(G), B(G), C(G) \in \mathbb{H}(G)$. Let $\alpha \subset \{ 1, 2, \dots, n \}$ such that $A(G)[\alpha]$ be a fully specified principal submatrix. Clearly, $M=(A(G)-B(G))[\alpha]$ and $N=(B(G)-C(G))[\alpha]$ are fully specified principal submatrices of $A(G)-B(G)$ and $B(G)-C(G)$, respectively. Since $A(G)-B(G)$ and $B(G)-C(G)$ are partial positive semidefinite, $M$ and $N$ are positive semidefinite. Note that $M+N$ is a fully specified principal submatrix of $A(G)-C(G)$ and positive semidefinite. Since $\alpha$ is arbitrary, every fully specified submarix of $A(G)-C(G)$ is positive semidefinite, implying $A(G)\leq C(G)$.
\end{proof}

\begin{example}\label{ex:PD-order}
Consider the following $G$-partial (positive) matrices:
\begin{equation*}
A(G) =
\begin{bmatrix}
    3 & ? & 2 & 1\\
    ? & 6 & 1 & ?\\
    2 & 1 & 4 & 1\\
    1 & ? & 1 & 5
\end{bmatrix}, \quad
B(G) =
\begin{bmatrix}
    1 & ? & 1 & 1\\
    ? & 5 & 1 & ?\\
    1 & 1 & 3 & 1\\
    1 & ? & 1 & 2
\end{bmatrix}.
\end{equation*}
Since the difference
\begin{equation*}
C(G):=A(G)-B(G) =
\begin{bmatrix}
    2 & ? & 1 & 0\\
    ? & 1 & 0 & ?\\
    1 & 0 & 1 & 0\\
    0 & ? & 0 & 3
\end{bmatrix}
\end{equation*}
is partial positive definite,
by the definition $A(G) > B(G)$.
\end{example}

\begin{remark}
In general, the existence of positive completions of two partial matrices $A(G)$ and $B(G)$ does not guarantee the existence of positive completion of $A(G)-B(G)$.
For example, the partial matrices $A(G)$ and $B(G)$ in Example \ref{ex:PD-order}
have positive definite completions (see Example \ref{completable_example1}).
Since the difference $C(G)$ is partial positive definite with positive completable graph $G$, it also has positive definite completions.
On the other hand, the partial matrix $A(G)-3B(G)$  does not have any positive (semi-)definite completion since it is not partial positive (semi-)definite although $A(G)\geq 0$, $3B(G)\geq 0$, and they have positive (semi-)definite completions.
\end{remark}

\begin{lemma}\label{lemma_order}
Let $G$ be a given completable graph. Suppose that $A(G)$ and $B(G)$ are partial positive definite matrices. Then
$\mathfrak{p}^{+}[A(G)-B(G)] \subset \mathfrak{p}^{+}[A(G)] - \mathfrak{p}^{+}[B(G)]$.
\begin{proof}
Let $G = (V,E)$ be a given completable graph. If $\mathfrak{p}^{+}[A(G)-B(G)] = \emptyset$, it is trivial.
Let $C \in \mathfrak{p}^{+}[A(G)-B(G)]$.
Then $C$ is a positive definite completion of $A(G)-B(G)$.
That is, there exist a positive definite matrix $C=[c_{ij}]$ such that $c_{ij}=a_{ij}-b_{ij}$ for all $\{i,j\} \in E$, where $A(G)=[a_{ij}]_{\{i,j\}\in E}$ and $B(G)=[b_{ij}]_{\{i,j\}\in E}$.
Let $B=[b_{ij}] \in \mathfrak{p}^{+}[B(G)]$. Since $B>0$ and $C>0$, we have $B+C>0$. Since $b_{ij}+c_{ij}= a_{ij}$ for all $\{i,j\} \in E$, it follows that $B+C \in \mathfrak{p}^{+}[A(G)]$. Since $B$ is arbitrary, it holds that $B+C \in \mathfrak{p}^{+}[A(G)]$ for all $B \in \mathfrak{p}^{+}[B(G)]$.
Thus, $C=(B+C)-B \in \mathfrak{p}^{+}[A(G)] - \mathfrak{p}^{+}[B(G)]$.
\end{proof}
\end{lemma}

\begin{remark}
\begin{itemize}
\item[(i)] For a completable graph $G$, assume that $A(G)$ and $B(G)$ are partial positive semi-definite matrices. Then $\mathfrak{p}[A(G)-B(G)] \subset \mathfrak{p}[A(G)] - \mathfrak{p}[B(G)]$ by following the proof of Lemma \ref{lemma_order} similarly to positive semi-definite completions.

\item[(ii)] Since $\mathfrak{p}^{+}[A(G)] - \mathfrak{p}^{+}[B(G)]$ may or may not include an element which is not a positive definite matrix,
in general, $\mathfrak{p}^{+}[A(G)-B(G)] \neq \mathfrak{p}^{+}[A(G)] - \mathfrak{p}^{+}[B(G)]$.
For example, consider the following partial positive definite matrices.
\begin{equation*}
A(G)=
\begin{bmatrix}
2.5 & 1 & ?\\
1 & 2.5 & 1\\
? & 1 &  2.5\\
\end{bmatrix},
\quad
B(G)=
\begin{bmatrix}
2 & 1 & ?\\
1 & 2 & 1\\
? & 1 &  2\\
\end{bmatrix}
\end{equation*}
Then it is clear that
\begin{equation*}
A=
\begin{bmatrix}
2.5 & 1 & 1\\
1 & 2.5 & 1\\
1 & 1 &  2.5\\
\end{bmatrix}
\in \mathfrak{p}[A(G)],
\quad
B=
\begin{bmatrix}
2 & 1 & 0\\
1 & 2 & 1\\
0 & 1 &  2\\
\end{bmatrix}
\in \mathfrak{p}[B(G)],
\end{equation*}
and
\begin{equation*}
A-B=
\begin{bmatrix}
0.5 & 0 & 1\\
0 & 0.5 & 0\\
1 & 0 &  0.5\\
\end{bmatrix}
\in \mathfrak{p}^{+}[A(G)] - \mathfrak{p}^{+}[B(G)].
\end{equation*}
However, $A-B \notin \mathfrak{p}^{+}[A(G)-B(G)]$ since $A-B$ is not positive definite.
This shows that $A(G)<B(G)$ does not imply $A<B$ for all $A\in \mathfrak{p}^+[A(G)]$ and $B\in \mathfrak{p}^+[B(G)]$.
\end{itemize}
\end{remark}

\begin{theorem}
Let $G$ be a given completable graph. Suppose that $A(G)$ and $B(G)$ are $G$-partial matrices with $0 < B(G) < A(G)$. Then
$$ \mathfrak{p}^{+}[A(G)-B(G))] = (\mathfrak{p}^{+}[A(G)]-\mathfrak{p}^{+}[B(G)]) \cap \mathbb{P}. $$
That is, all positive completions of $A(G)-B(G)$ are expressed as differences of positive completions of $A(G)$ and $B(G)$.
\end{theorem}

\begin{proof}
Let $C\in \mathfrak{p}^{+}[A(G)-B(G)]$. Then by Lemma \ref{lemma_order} it follows that $C \in \mathfrak{p}^{+}[A(G)] - \mathfrak{p}^{+}[B(G)]$. Since $C$ is positive definite, it follows that $C \in (\mathfrak{p}^{+}[A(G)] - \mathfrak{p}^{+}[B(G)])\cap \mathbb{P} $, so $\mathfrak{p}^{+}[A(G)-B(G)] \subset
(\mathfrak{p}^{+}[A(G)] - \mathfrak{p}^{+}[B(G)])\cap \mathbb{P}$. Let $D \in (\mathfrak{p}^{+}[A(G)] - \mathfrak{p}^{+}[B(G)]) \cap \mathbb{P}$. Then there exist $M \in \mathfrak{p}^{+}[A(G)]$ and $N \in \mathfrak{p}^{+}[B(G)]$ such that $D = M-N$ is positive definite.
Since $M-N \in \mathfrak{p}^{+}[A(G)-B(G)]$, we have that $D \in \mathfrak{p}^{+}[A(G)-B(G)]$.
\end{proof}

\section{Maximizing the determinant}\label{sec:maxdet}

For a given completable graph $G$, we see in this section several interesting consequences for positive definite completions of $G$-partial positive matrices maximizing the determinant with the previous notions of geometric mean and order relation.

\begin{lemma}\cite[Lemma 1]{GJSW84}\label{lemma:concave}
The function $f(A)=\log{\det{(A)}}$ is strictly concave on $\mathfrak{p}^{+}[A(G)]$.
\end{lemma}


\begin{theorem}\label{theorem:maxdet}
Let $G$ be a given completable graph, and let $A(G)$ be a $G$-partial matrix with $A(G) > 0$. Then there exists a unique positive definite completion of $A(G)$, say $\widehat{A}$, such that
\begin{equation*}
\det(\widehat{A}) = \max{\{ \det(M): M \in \mathfrak{p}^{+}[A(G)] \}}.
\end{equation*}
Furthermore, $\widehat{A}$ is the unique positive definite completion of $A(G)$ whose inverse $C = [c_{ij}]$ satisfies
$$ c_{ij}=0 \quad \textit{for all } \{ i,j \} \notin E. $$
\end{theorem}

\begin{proof}
It follows by Lemma \ref{lemma:concave} amd Theorem \ref{T:property} (see \cite{GJSW84}).
\end{proof}

Now we investigate relationship between $\widehat{A}$ and other positive definite completions for their determinants.

\begin{theorem}
Let $G$ be a given completable graph, and let $A(G)$ be a $G$-partial matrix with $A(G) > 0$. For $0 < k < \det(\widehat{A})$, there exists $A \in \mathfrak{p}^{+}[A(G)]$ such that $\det(A) = k$.
\end{theorem}


\begin{proof}
Since $\mathfrak{p}[A(G)]$ is convex, it is also path-connected, hence connected.
Since the determinant,
$\det:\mathfrak{p}[A(G)] \rightarrow \R$, is a continuous function on $\mathfrak{p}[A(G)]$, the range of the determinant function is connected.
By the inequality of Hadamard, it follows that
$$0 \leq \det(A) \leq \prod_{i=1}^n a_{ii} \quad \text{for all } A\in \mathfrak{p}[A(G)]. $$
So,
$\det(\mathfrak{p}[A(G)])$ is bounded in $\R$, i.e., $\det(\mathfrak{p}[A(G)]) = [a,b]$ for some $0 \leq a< b$.
Claim that $a=0$.
It is enough to show that there exist $S\in \mathfrak{p}[A(G)]$ such that $\det(S)=0$.
%

Let $Z\in\mathfrak{p}[A(G)]$ with $\det{(Z)}\neq 0$. Pick one missing entry of $A(G)$, say $x$. We fix all entries of $Z$ except $x$ and convert $Z$ into the block matrix of the form \eqref{block_1missing} via permutation similarity to place $x$ in the $(1,n)$ spot.
Set $\tilde{Z}=[\tilde{z}_{ij}]$ as the matrix obtained by taking $x=\sqrt{ab}$, where $a,b$ are entries in the $(1,1)$ and $(n,n)$ spots, respectively.  Since the determinant of the principal submatrix $\tilde{Z}[1,n]$ is zero,  $\det{(\tilde{Z})}=0$.
By Proposition 2.3 in \cite{Curto1993}, $\tilde{Z}\geq 0$, so $\tilde{Z} \in \mathfrak{p}[A(G)]$.
\end{proof}


Using the similar proof of \cite[Theorem 1]{DJ05}, one can have an integral representation for determinants of two positive definite completion of a $G$-partial matrix.

\begin{theorem}\label{theorem:det_relationship}
Let $\mathcal{C}$ be a nonempty convex subset of $\mathbb{P}$.
Let ${A}_0$ and ${A}_1$ be in $\mathcal{C}$. Define
${A}(\lambda) :=(1-\lambda){A}_0 + \lambda {A}_1 \quad \text{for all } \lambda \in [0,1].$
Then,
\begin{equation*}
\det({A}_1)=\det({A}_0)  \exp{\left\{ \int_0^1 \tr({A}(\lambda)^{-1} ({A}_1-{A}_0)) d\lambda \right\}}.
\end{equation*}
\end{theorem}
\begin{proof}
Let ${M}(\lambda)=[m_{ij}(\lambda)]_{ij}$ be a $n\times n$ Hermitian matrix whose entries $m_{ij}(\lambda)$ are functions of a parameter $\lambda$ on $I=[a,b]$ and its determinant be denoted by $\Delta(\lambda)$.
Assume that the functions $m_{ij}(\lambda)$ are differentiable on $I$ for all $i,j$ and $\det{({M}(\lambda))} \neq 0$ for all $\lambda \in I$.
Then by the trace theorem \cite[p. 83]{LAN02},
it follows that
\begin{equation}\label{trace}
\frac{d \Delta(\lambda)}{d\lambda} = \Delta(\lambda)
\tr\bigg( {M}(\lambda)^{-1}\frac{d {M}(\lambda)}{d\lambda} \bigg),
\end{equation}
where
$\dfrac{d{M}(\lambda)}{d\lambda}$
is the matrix whose elements are $\dfrac{d(m_{ij}(\lambda))}{d\lambda}$.
Clearly, ${A}(\lambda) \subset\mathcal{C} \subset \mathbb{P}$ for all $\lambda\in[0,1]$.
Since $\dfrac{d{A}(\lambda)}{d\lambda} = {A}_1 - {A}_0$, by \eqref{trace} it follows that
\begin{equation}\label{trace1}
\frac{d \Delta(\lambda)}{d\lambda} = \Delta(\lambda)
\tr\bigg( {A}(\lambda)^{-1} ({A}_1-{A}_0) \bigg).
\end{equation}
Since $\Delta(\lambda) \neq 0$, the equation \eqref{trace1} can be rewritten as
\begin{equation}\label{trace2}
\frac{d \Delta(\lambda)}{\Delta(\lambda)} =
\tr\bigg( {A}(\lambda)^{-1} ({A}_1-{A}_0) \bigg) d\lambda.
\end{equation}
Since ${A}(\lambda)^{-1}$ is continuous, $\tr{({A}(\lambda)^{-1} ({A}_1-{A}_0))}$ is a continuous function of $\lambda$ on $[0,1]$. Also,
$\Delta(\lambda)$ is a polynomial of $\lambda$.
By integrating both sides of the equation \eqref{trace2} we have
\begin{equation*}
\ln{ (|\Delta(1)/\Delta(0)| )}=
\int_0^1 \tr({A}(\lambda)^{-1} ({A}_1-{A}_0) ) d\lambda.
\end{equation*}
Note that $\Delta(0)=\det({A}_0)$ and $\Delta(1)=\det({A}_1)$.
\end{proof}

Consider a zero-mean, multivariate Gaussian distribution on $\R^n$,
\begin{equation*}
f({x})=(2\pi)^{-n/2} |{\Sigma}|^{-1/2} \exp{\bigg\{-\frac{1}{2} {x}^\top {\Sigma}^{-1} {x}\bigg\}},
\end{equation*}
where the covariance matrix ${\Sigma}(G)$ is partial positive definite with a completable graph $G$.
Recall that $H(f)$ in \eqref{e:entropy} is the Shannon entropy.
\begin{theorem}
Let $f_0,f_1$ be zero-mean, mutivariate Gaussian distributions on $\R^n$ with covariance matrices
${\Sigma}_0, {\Sigma}_1 \in \mathfrak{p}^{+}[{\Sigma}(G)]$, respectively.
Then the following are true:
\begin{itemize}
\item[(1)] the difference between Shannon entropy for two distributions can be expressed as
\begin{equation*}
    H(f_1)-H(f_0) = \frac{1}{2} \int_0^1 \tr({\Sigma}(\lambda)^{-1} ({\Sigma}_1-{\Sigma}_0)) d\lambda,
\end{equation*}
where ${\Sigma}(\lambda) = (1-\lambda){\Sigma}_0 +\lambda {\Sigma}_1$.
\item[(2)] for a zero-mean, multivariate Gaussian distribution on $\R^n$ with the covariance matrix ${\Sigma}_0 \#_t {\Sigma}_1$, says $f_0 \#_t f_1$, it holds that
\begin{equation*}
H(f_0 \#_t f_1) = (1-t)H(f_0) + tH(f_1).
\end{equation*}
\end{itemize}

\end{theorem}
\begin{proof}
(i) By Eq. \eqref{e:entropy} and Theorem \ref{theorem:det_relationship}, it holds.
(ii) by Eq. \eqref{e:entropy} and Theorem \ref{T:geomean} (9), it follows that
\begin{align*}
    H(f_0 \#_t f_1)
    &= \frac{1}{2}\log(\det{{\Sigma}_0 \#_t {\Sigma}_1})+\frac{1}{2}n(1+ \log{(2\pi))}\\
    &= \frac{1}{2}((1-t)\log(\det{{\Sigma}_0}) + t\log(\det{{\Sigma}_1}))
    +\frac{1}{2}n(1+ \log{(2\pi))}\\
    &= (1-t)H(f_0) + tH(f_1).
\end{align*}
\end{proof}
Especially, when $t=1/2$, it shows that entropy of distribution with the geometric mean of two covariances is the average of entropy of two distributions with each covariance.

\begin{corollary}
Let ${A}(G_1)>0$ and ${B}(G_2)>0$ be given partial matrices with completable graphs $G_1$ and $G_2$. Suppose that ${A}(G_1) \#_{t} {B}(G_2)$ is convex.
Then
the determinant of ${A} \#_{t} {B}$ in ${A}(G_1) \#_{t} {B}(G_2)$ can be expressed as
\begin{equation*}
\det(\widehat{{A}})^{1-t} \det(\widehat{{B}})^{t} \cdot \exp{\left\{ \int_0^1 \tr({S}_{t}(\lambda)^{-1} {T}_{t}) d\lambda \right\}},
\end{equation*}
where ${S}_{t}(\lambda)=    (1-\lambda)(\widehat{{A}} \#_{t} \widehat{{B}}) +\lambda ( {A} \#_{t} {B})$  and ${T}_{t} = \widehat{{A}} \#_{t} \widehat{{B}} - {A} \#_{t} {B}$.
\end{corollary}

\begin{theorem}[Fischer's Inequality\cite{HJ}]\label{thm:Fischer}
Let a $(m+n)\times (m+n)$ Hermitian positive definite matrix $H$ have the partitioned form
\begin{equation*}
H=
\begin{bmatrix}
A & X \\
X & B
\end{bmatrix},
\end{equation*}
where $A\in M_m$ and $B\in M_n$. Then
$$\det{H}\leq (\det{A})(\det{B})$$
with equality if and only if $X=0$.
\end{theorem}

The following is a similar result to Fischer's Inequality for partial matrices.
\begin{proposition} \label{prop:maxdet}
Let $G_1 = (V_1,E_1)$ and $G_2 = (V_2,E_2)$ be disjoint completable graphs with $V_1\cap V_2 =\emptyset $. Let $G = (V,E)$ with $V = V_1 \cup V_2$ and $E = E_1 \cup E_2$.
Let
\begin{equation*}
H(G)=
\begin{bmatrix}
A(G_1) & X \\
X & B(G_2)
\end{bmatrix},
\end{equation*}
where all entries of $X$ are missing.
If $A(G_1)>0$ and $B(G_2)>0$,
then
$$\det{H}\leq (\det{\widehat{A}})(\det{\widehat{B}}) \quad \text{for all }H\in \mathfrak{p}[A(G)] $$
with equality if and only if $X=\mathbf{0}$.
Here $\widehat{A}$ and $\widehat{B}$ are the maximum determinant positive definite completions of $A(G_1)$ and $B(G_2)$, respectively.
\end{proposition}
\begin{proof}
Since $G_1$ and $G_2$ are chordal, so is $G$.
By Theorem \ref{chordal},
the graph $G$ is completable. Since $A(G_1)>0$ and $B(G_2)>0$ , it is clear that $H(G)>0$ , implying $\mathfrak{p}^+[H(G)] \neq \emptyset$.
By Theorem \ref{theorem:maxdet}, the graphs $A(G_1)$, $B(G_2)$ have the maximum determinant positive definite completions, say $\widehat{A}$, $\widehat{B}$ respectively.
By Theorem \ref{thm:Fischer}, it follows that
$\det(H) \leq (\det{A})(\det{B})\leq (\det{\widehat{A}})(\det{\widehat{B}}) $
for all $H\in \mathfrak{p}[H(G)]$, $A\in \mathfrak{p}[A(G_1)]$, and $B\in \mathfrak{p}[B(G_2)]$. When $X=\mathbf{0}$, the equality holds. If $\det(H) = (\det{\widehat{A}})(\det{\widehat{B}}) $, then $X=\mathbf{0}$ by Theorem \ref{thm:Fischer}.
\end{proof}

\begin{lemma} \label{L:joint concave}
The map $f: \mathbb{P} \times \mathbb{P} \to \mathbb{P}$ defined by $f(A, B) = \log \det (A \#_{t} B)$ for any $t \in [0,1]$ is strictly jointly concave.
\end{lemma}

\begin{proof}
By Theorem \ref{T:geomean} (1), the function $f$ on $\mathbb{P} \times \mathbb{P}$ can be written as
\begin{displaymath}
f(A, B) = \log \det (A \#_{t} B) = (1-t) \log \det A + t \log \det B.
\end{displaymath}
Since the map $\log \det: \mathbb{P} \to \mathbb{R}$ is strictly concave, so is $f$.
\end{proof}

\begin{remark}
We can show the joint concavity of the map $f$ in Lemma \ref{L:joint concave} by using the joint concavity of geometric mean. Indeed, by Theorem \ref{T:geomean} (7),
\begin{equation*}
[(1 - \lambda) A_{1} + \lambda A_{2}] \#_{t} [(1 - \lambda) B_{1} + \lambda B_{2}] \geq (1 - \lambda) (A_{1} \#_{t} B_{1}) + \lambda (A_{2} \#_{t} B_{2})
\end{equation*}
for any $A_{1}, A_{2}, B_{1}, B_{2} \in \mathbb{P}$ and any $\lambda \in [0,1]$. Since $0< A \leq B $ implies $0<\det A \leq \det B$ from \cite[Corollary 7.7.4]{HJ}, the logarithmic map $\log: (0, \infty) \to \mathbb{R}$ is monotone increasing, and the map $\log \det: \mathbb{P} \to \mathbb{R}$ is strictly concave,
\begin{equation*}
\begin{split}
f((1 - \lambda) (A_{1}, B_{1}) + \lambda (A_{2}, B_{2})) & = \log \det [(1 - \lambda) A_{1} + \lambda A_{2}] \#_{t} [(1 - \lambda) B_{1} + \lambda B_{2}] \\
& \geq \log \det [ (1 - \lambda) (A_{1} \#_{t} B_{1}) + \lambda (A_{2} \#_{t} B_{2}) ] \\
& \geq (1 - \lambda) f(A_{1}, B_{1}) + \lambda f(A_{2}, B_{2}).
\end{split}
\end{equation*}
\end{remark}

\begin{theorem} \label{T:maxAG}
Let $G$ and $F$ be a given completable graph, and let $A(G)$ and $B(F)$ be partial matrices with $A(G) > 0$ and $B(F) > 0$. For $t \in [0,1]$ there exists a unique positive definite completion
$H$ of $A(G) \#_{t} B(F)$ such that
\begin{equation*}
\det(H)=
\max{\{ \det(M \#_{t} N): M \#_{t} N \in A(G)\#_{t}B(F) \}}.
\end{equation*}
Furthermore,
$H = \widehat{A} \#_{t} \widehat{B},
$
where $\widehat{A}$ and $\widehat{B}$ are the maximum determinant positive definite completions of $A(G)$ and $B(F)$, respectively.
\end{theorem}

\begin{proof}
Since $\det(M) \leq \det(\widehat{A})$ and $\det(N) \leq \det(\widehat{B})$ for all $M \in \mathfrak{p}^{+}[A(G)]$ and $N \in \mathfrak{p}^{+}[B(F)]$, it is clear that $\det(M \#_t N) = (\det(M))^{1-t} (\det(N))^t \leq (\det(\widehat{A}))^{1-t} (\det(\widehat{B}))^t = \det(\widehat{A} \# \widehat{B})$. We just show the uniqueness.
Suppose that there exists $C \in \mathfrak{p}^{+}[A(G)]$ and $D \in \mathfrak{p}^{+}[B(F)]$ such that
$\det{(C \#_{t} D)} = \det(\widehat{A} \#_{t} \widehat{B})$.
Then,
$(\det(C))^{1-t} (\det(D))^{t} = (\det(\widehat{A}))^{1-t} (\det(\widehat{B}))^{t}$.
Setting
$x=\det(C)/\det(\widehat{A})$, $y=\det(D)/\det(\widehat{B})$, we have $x^{1-t} y^t = 1$, implying that $(t-1)\log{x} = t\log{y}$. Since $x\leq 1$ and $y\leq 1$, it must be $x=y=1$ for some $0<t<1$. Therefore, $\det(\widehat{A})=\det(C)$ and $\det(\widehat{B})=\det(D)$. By the uniqueness of $\widehat{A}$ and $\widehat{B}$, it holds that $\widehat{A}=C$ and $\widehat{B}=D$.
\end{proof}

In other words, the positive definite completion of $A(G) \#_{t} B(F)$ can uniquely be expressed as positive definite completions of $A(G)$ and $B(F)$ with respect to maximum determinant.

\begin{corollary}
Let $G_1=(V_1,E_1)$ and $G_2=(V_2,E_2)$ be disjoint completable graphs with $V_1\cap V_2 = \emptyset$.
Let $G = (V,E)$ with $V = V_1 \cup V_2$ and $E = E_1 \cup E_2$.
Let
\begin{equation*}
A(G) =
\begin{bmatrix}
A_{1}(G_{1}) & X \\
X & A_{2}(G_{2})
\end{bmatrix}
\text{ and }
B(G) =
\begin{bmatrix}
B_{1}(G_{1}) & Y \\
Y & B_{2}(G_{2})
\end{bmatrix},
\end{equation*}
where all entries of $X$ and $Y$ are missing.
If $A_{i}(G_{i})>0$ and $B_{i}(G_{i})>0$ for $i = 1, 2$,
then
$$\det{(A\#_{t} B)}\leq \det{(A_1\#_{t} B_1)} \det{(A_2\#_{t} B_2)}
\quad \text{for all }A\#_t B\in A(G)\#_{t} B(G) $$
with equality if and only if $X=Y=\mathbf{0}$.
\end{corollary}

\begin{proof}
By Proposition \ref{prop:maxdet} and Theorem \ref{T:geomean} (9) it is trivial.
\end{proof}

\begin{theorem} \label{T:order_property}
Let $G$ be a given completable graph.
If $0 < A(G) \leq B(G)$, then $0 < \det(\widehat{A}) \leq \det(\widehat{B})$.
\end{theorem}

\begin{proof}
Since the graph $G=(V,E)$ is completable and $B(G)-A(G)>0$, the partial matrix $B(G)-A(G)$ has the maximum determinant positive definite completion, say $\widehat{M}$.
Let $\widehat{A}$ be the maximum determinant positive definite completion of $A(G)$. Let $b_{ij}$ be entries of $B(G)$ for $\{i,j\}\in E$.
Since $\widehat{M}_{ij}+\widehat{A}_{ij}=b_{ij}$ for all $\{i,j\}\in E$,
$\widehat{M}+\widehat{A}$ is a positive definition completion of $B(G)$. Then it follows that
\begin{equation*}
0< \det{(\widehat{A})} <
\det{(\widehat{M})}+\det{(\widehat{A})} \leq
\det{(\widehat{M}+\widehat{A})} \leq
\det{(\widehat{B})}
\end{equation*}
Note that $\det{(A)}+\det{(B)} \leq \det{(A+B)}$ for $A,B\in \mathbb{P}$ (see p.511, \cite{HJ}).
\end{proof}

\section{Computational results}\label{sec:computation}

Consider finding the maximum determinant positive definite completion among $A(G)\#B(G)$ when $A(G)$ and $B(G)$ are $n\times n$ $G$-partial positive definite matrices with only one missing entry in the $(1,n)$ position, respectively.

\begin{theorem}[\cite{MR1059486}]\label{theorem:maxdetone}
Consider the following partial matrix with the only one missing entry:
\begin{equation}\label{block_1missing}
    H(x)=
    \begin{bmatrix}
    a & v^T & x\\
    v & C & w\\
    x & w^T & b
    \end{bmatrix},
\end{equation}
where all entries are given except $x$.
If $H(x)$ is partial positive definite, $H(x)$ has a positive definite completion. Indeed, the set of all such completions is given by the inequality
\begin{equation*}
    |x-v^TC^{-1}w|^2 < \dfrac{\det{A}\det{B}}{(\det{C})^2}.
\end{equation*}
Two endpoints of this interval give singular positive semidefinite completions of $H(x)$.
When $x=v^T C^{-1}w $,
the positive definite completion has the maximum determinant
\begin{equation*}
    \frac{\det{(A)}\det{(B)}}{(\det{C})^2},
\end{equation*}
where
\begin{equation*}
A=
    \begin{bmatrix}
  a & v^T\\
  v & C
    \end{bmatrix}
    \quad\text{and}\quad
    B=
    \begin{bmatrix}
  C & w\\
  w^T & b
    \end{bmatrix}.
\end{equation*}
\end{theorem}
Clearly it holds that $H(x)>0$ if and only if $A>0$ and $B>0$.
In \cite{GHJT99} a robust and fast algorithm based on the preceding theorem is introduced. Suppose that a partial matrix with one or possibly more then one missing entries is given.
We fix all but one entry and then place the position in the $(1,n)$ spot via permutation similarity. Then by Theorem \ref{theorem:maxdetone} the maximum determinant completion can be found. Repeating this process, the sequence of completion matrices is constructed with respect to each of missing entries. It is shown that the sequence
converges to the unique global maximum determinant completion (for more information, see \cite[Theorem 2]{GHJT99}).


Now using Theorem \ref{theorem:maxdetone} behaviors of $A(G)\#_{t}B(F)$ are shown computationally. All graphs in Figure 1--4 are generated by a MATLAB program \cite{MATLAB2015b}.
\begin{example}
Consider the following partial matrices.
\begin{equation*}
A(x)=
    \begin{bmatrix}
    3 & -1 & x\\
    -1 & 3 & 2\\
    x & 2 & 4
    \end{bmatrix},
\quad
B(y)=
    \begin{bmatrix}
    4 & 3 & y\\
    3 & 5 & -1\\
    y & -1 & 2
    \end{bmatrix},
\end{equation*}
where $x$ and $y$ are missing entries.
Since $A(x)>0$ and $B(y)>0$, by Theorem \ref{theorem:maxdetone} they have positive definite completions when $-10/3 < x < 2$ and $(-3\sqrt{11}-3)/5< y < (3\sqrt{11}-3)/5$. The determinant and each eigenvalue of $A(x)\#B(y)$ with respect to such values $x$ and $y$ are shown in Figure \ref{fig1}.  Also, it is shown that $A(x)\#B(y)$ have the maximum determinants when $x=-2/3$ and $y=-3/5$, respectively.

\begin{figure}
\centering
\includegraphics[scale=0.6]{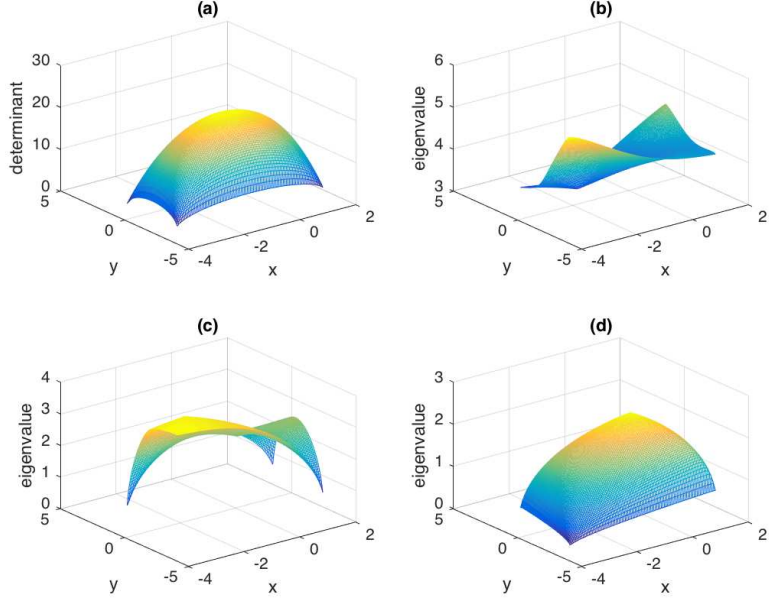}
\caption{(a) the determinant of $A(x)\#B(y)$; (b) the maximum eigenvalue of $A(x)\#B(y)$; (c) the second maximum eigenvalue of $A(x)\#B(y)$; (d) the smallest eigenvlue of $A(x)\#B(y)$.}\label{fig1}
\label{fig:testfig}
\end{figure}
\end{example}

\begin{example}
We consider the geometric mean of same partial matrices.
Let
\begin{equation*}
A(x)=
    \begin{pmatrix}
    3 & -1 & 1 & 1 & x\\
    -1 & 3 & -1 & 1 & 0\\
    1 & -1 & 3 & 2 & 1\\
    1 & 1 & 2 & 4 & 2\\
    x & 0 & 1 & 2 & 4
    \end{pmatrix} \text { and }
B(y)=
    \begin{pmatrix}
    3 & 0 & 1 & 2 & y\\
    0 & 1 & 0 & -1 & 0\\
    1 & 0 & 5 & -1 & 1\\
    2 & -1 & -1 & 3 & 0\\
    y & 0 & 1 & 0 & 4
    \end{pmatrix},
\end{equation*}
where $x$ and $y$ are missing entries.
Since $A(x)>0$ and $B(y)>0$, by Theorem \ref{theorem:maxdetone} they have positive definite completions when $(10-\sqrt{1036})/3 < x < (10+\sqrt{1036})/3$ and $(4-\sqrt{34})/9 < y < (4+\sqrt{34})/9$. The determinant and each eigenvalue of $A(x)\#B(y)$ with respect to such values $x$ and $y$ are shown in Figure \ref{fig2}.  Also, it is shown that $A(x)\#B(y)$ have the maximum determinants when $x=10/13$ and $y=4/9$, respectively.

\end{example}

\begin{figure}
  \centering
    \includegraphics[width=1\textwidth]{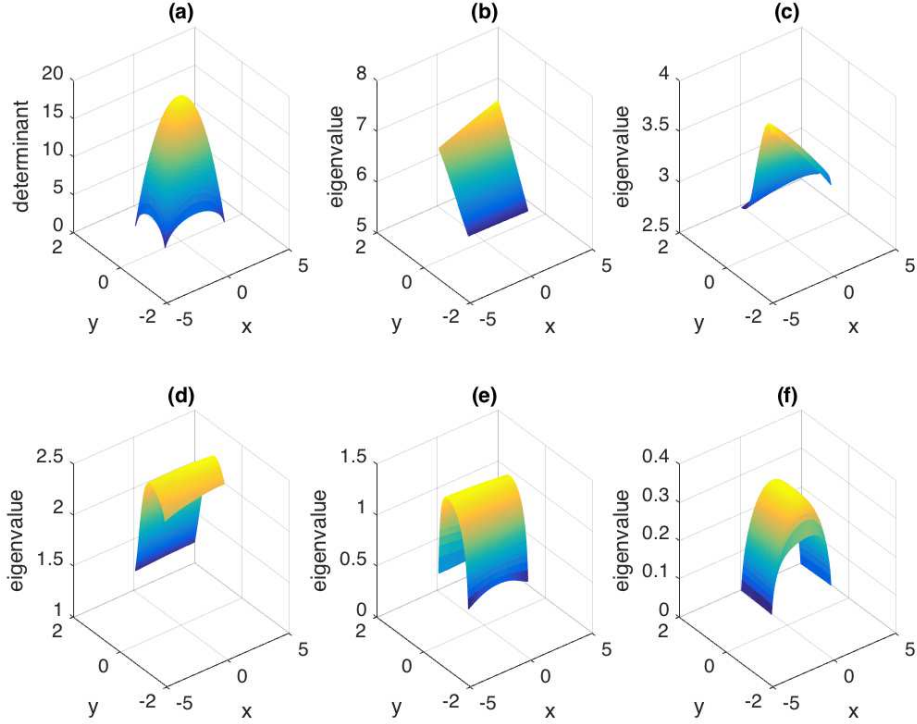}
     \caption{(a) the determinant of $A(x)\#B(y)$; (b) the largest eigenvalue of $A(x)\#B(y)$; (c) the second largest eigenvalue of $A(x)\#B(y)$; (d) the third largest eigenvlue of $A(x)\#B(y)$; (e) the fourth eigenvlue of $A(x)\#B(y)$; (f) the smallest eigenvlue of $A(x)\#B(y)$.}\label{fig2}
\end{figure}

\begin{figure}
     \centering
    \includegraphics[scale=0.6]{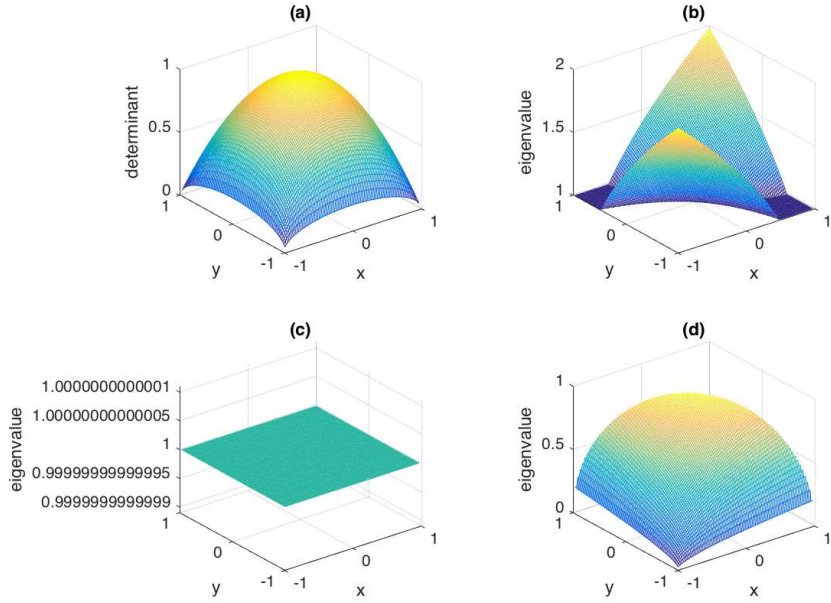}
     \caption{(a) the determinant of $I_{10}(x)\#I_{10}(y)$; (b) the largest eigenvalue of $I_{10}(x)\#I_{10}(y)$; (c) the second largest eigenvalue of $I_{10}(x)\#I_{10}(y)$; (d) the smallest eigenvlue of $I_{10}(x)\#I_{10}(y)$.}\label{fig3}
\end{figure}

\begin{example}
Let
\begin{equation*}
I_n(x)=
    \begin{pmatrix}
    1 & 0 & x\\
    0 & I & 0\\
    x & 0 & 1
    \end{pmatrix},
\end{equation*}
where $I$ is the $(n-2)\times (n-2)$ identity matrix.
By Theorem \ref{theorem:maxdetone}, $I_n(x)\#I_n(y)$ has positive definite completions when $x,y\in(-1, 1)$ and has the maximum determinants when $x=y=0$. The determinant and each eigenvalue of $I(x)\#I(y)$ with respect to such values $x$ and $y$ are shown in Figure \ref{fig3}. Note that each graph looks like symmetric with respect to $y=-x$ and $y=x$ since $I_n(x)\#I_n(y)=I_n(y)\#I_n(x)=I_n(-x)\#I_n(-y)$.
\end{example}

\begin{example}
Let
\begin{equation*}
A(x,y)=
    \begin{pmatrix}
    2 & 1 & x\\
    1 & 2 & y\\
    x & y & 2
    \end{pmatrix},
    \quad
    B=
    \begin{pmatrix}
    4 & 3 & 0\\
    3 & 5 & -1\\
    0 & -1 & 2
    \end{pmatrix},
\end{equation*}
By simple calculations, it can be shown that $A(x,y)\#B$ has positive definite completions if and only if $|x|<2$, $|y|<2$, and $6+3xy-2x^2-2y^2>0$, and has the maximum determinants when $x=y=0$. The determinant and each eigenvalue of $A(x,y)\#B$ with respect to such values $x$ and $y$ are shown in Figure \ref{fig4}.
\begin{figure}
  \centering
    \includegraphics[scale=0.5]{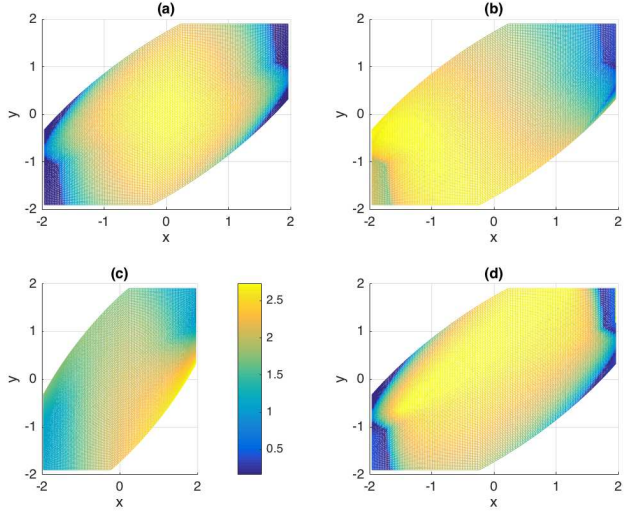}
     \caption{(a) the determinant of $A(x,y)\#B$; (b) the largest eigenvalue of $A(x,y)\#B$; (c) the second largest eigenvalue of $A(x,y)\#B$; (d) the smallest eigenvlue of $A(x,y)\#B$.}\label{fig4}
\end{figure}
\end{example}

\begin{example}
Consider the following positive definite matrix. For sufficiently small $\epsilon>0$,
\begin{equation*}
A=
    \begin{pmatrix}
1.5 & 1 & 1\\
1 & 1 & 1/3+\epsilon\\
1 & 1/3+\epsilon & 1\\
    \end{pmatrix}
\end{equation*}
Setting $a_{23}=0.33333333$, the matrix $A$ will lose positivity. So, arbitrarily small perturbations of a positive definite matrix eject one from the cone of positive definite matrices. Thus, for any positive definite matrix $B$, $A\# B$ will be changed fast as even very small perturbation occurs.
\end{example}

\section{Application in Computer Vision}\label{sec:application}
Let $I$ be a $1$-dimensional intensity or $3$-dimensional color image and $F$ be the feature image extracted from $I$. For a given rectangular region $R\subset F$, let $\{{z}_k\}_{1\leq k\leq m}$ be the $d$-dimensional feature points on $R$. The \emph{Region covariance(RC) descriptor} is the $d \times d$ covariance matrix of the feature points which is defined by
\begin{equation}
{C}_R = \frac{1}{m-1} \sum_{k=1}^m ({z}_k-{\mu})({z}_k-{\mu})^\top,
\end{equation}
where ${\mu}$ is the mean of the points (see \cite{MSR15}).

The RC descriptor has recently become a popular method in several areas such as computer vision and applications of these topics to problems in optimization, machine learning, medical image, and etc \cite{Tuzel2006}.
The RC descriptors are symmetric positive definite matrices which is relatively low-dimensional descriptors extracted from several different features computed at the level of regions. Since a single covariance matrix extracted from a region is usually enough to match the region in different views and poses, RC descriptor consequently reduces the computational cost of classification.
In \cite{KSR14} an image classification scheme based on the generalized geometric mean of positive definite matrices computed from features of all sub-regions in a given medical image, specifically a breast histological image is proposed. Indeed, an image region $R$ can be divided into $n$ small non-overlapping sub-regions $\{R_1,\ldots,R_n\}$ to calculate the corresponding RC descriptors ${C}_{R_k}$, $k=1,\ldots,n$.
Note that the regional covariance descriptors computed from sub-image are points lying on the Riemannian manifold of positive definite matrices. Therefore,
a representative of different RC descriptors calculated from sub-images can be considered as the generalized geometric mean of positive definite matrices.
There are several different symmetric weighted geometric means for positive definite matrices, but we deal with the \emph{Karcher mean} as follows: for a positive probability vector ${\omega} = (w_{1}, \dots, w_{n})$
\begin{equation}\label{def:Karcher}
\Lambda({\omega};{A}_1,\ldots, {A}_n):=\argmin_{{X} \in \mathbb{P}} \sum_{k=1}^n w_{i} \delta^2({X},{A}_i),
\end{equation}
where $\delta(\cdot,\cdot)$ is defined in \eqref{eq:distancePD}.
It is shown that there exists  the unique minimum for the optimization problem \eqref{def:Karcher} if the matrices all lie in a convex ball in a Riemannian manifold (see \cite[section 6.15]{Berger03} and \cite{Kar}).
For more information, see \cite{BH06, Mo}.
Thus, the representative of RC descriptors for sub-regions can then be combined through the generalized geometric mean as $\Lambda({\omega};{C}_{R_1},\ldots, {C}_{R_n})$.

However, missing entries of RC descriptor in practice can occur due to various reasons, such as poor imaging quality or detector noise. Considering missing entries as the zero values or some values possibly reduces precision or encourage such matrix to lose its positivity. Assuming that ${C}_{R_1}(G_1),\ldots, {C}_{R_n}(G_n)$ are partial positive definite matrices with completable graphs $G_1,\dots,G_n$, the representative of RC descriptors for sub-regions with missing entries is the generalized geometric mean of partial positive definite matrices, which is $\Lambda({\omega};{C}_{R_1}(G_1),\ldots, {C}_{R_n}(G_n))$.

Here, we shortly introduce the recent result, called \emph{no dice theorem} \cite{Ho,LP14}, to compute the Karcher mean. For a positive probability vector ${\omega} = (w_{1}, \dots, w_{n})$, we denote
\begin{displaymath}
\displaystyle \overline{{\omega}} := (w_{1}, \ldots, w_{n}, w_{1}, \ldots, w_{n}, \ldots),
\end{displaymath}
and $\displaystyle s(N) := \sum_{i=1}^{N} \overline{\omega}_{i}$ for each $N \in \N$, where $\overline{\omega}_{i}$ is the $i$th component of the infinite-dimensional vector $\overline{{\omega}}$. The sequence of weighted inductive means is defined by
\begin{displaymath}
\displaystyle S_{1} = {A}_{1}, \ S_{N} = {A}_{k} \#_{\frac{s(N-1)}{s(N)}} S_{N-1}
\end{displaymath}
for natural numbers $N \geq 2$, where $k \in \{ 1, \dots, n \}$ is chosen so that $k \equiv N$ (mod $n$). Then
\begin{equation} \label{E:no-dice}
\displaystyle \lim_{N \to \infty} {S}_{N} = \Lambda({\omega}; {A}_{1}, \ldots, {A}_{n}).
\end{equation}
This is the special case of \emph{law of large numbers} on the Hadamard space of positive definite matrices. Using the convergence in \eqref{E:no-dice}, we can find approximately the Karcher mean of partial positive definite matrices to meet our needs.

\section{Final remarks}\label{sec:final_remarks}

We have studied the weighted geometric mean of two partial positive definite matrices including some numerical computation with missing entries. We finally close with some open problems arisen during our study. Let $G$ and $F$ be completable graphs.
\begin{itemize}
\item[(1)] For $A, B \in \mathbb{P}$, set
\begin{align*}
      A_{0} = A,~B_{0} = B,~
A_{n+1} = \left( \frac{A_{n}^{-1} + B_{n}^{-1}}{2} \right)^{-1},~B_{n+1} = \frac{A_{n} + B_{n}}{2}.
\end{align*}
It is known from \cite{LL01} that
\begin{displaymath}
A_{n} \leq A_{n+1} \leq A \# B \leq B_{n+1} \leq B_{n}
\end{displaymath}
for all $n \geq 1$, and the sequences $\{ A_{n} \}$ and $\{ B_{n} \}$ converge monotonically to $A \# B$. For subsets $\mathcal{S}$ and $\mathcal{T}$ of $\mathbb{P}$, we can define the harmonic mean and arithmetic mean such as
\begin{center}
$\displaystyle \left( \frac{\mathcal{S}^{-1} + \mathcal{T}^{-1}}{2} \right)^{-1}$ \ and \ $\displaystyle \frac{\mathcal{S} + \mathcal{T}}{2}$
\end{center}
via the natural definitions of scalar multiplication, sum, and inversion in Section 4. It is questionable that the sequences $\{ \mathcal{S}_{n} \}$ and $\{ \mathcal{T}_{n} \}$ of subsets of $\mathbb{P}$ constructed by the above mean iteration converge to $\mathcal{S} \# \mathcal{T}$. It may be applied to the geometric mean $\mathfrak{p}^{+}[A(G)] \# \mathfrak{p}^{+}[B(F)]$ of partial positive definite matrices $A(G)$ and $B(F)$.

\item[(2)] One can naturally ask the geometric characterization of the geometric mean of partial positive definite matrices. In Theorem \ref{T:property} and Remark \ref{R:property} we have seen that $\mathfrak{p}^{+}[A(G)]$ is nonempty, convex, and bounded. So $\mathfrak{p}^{+}[A(G)] \#_{t} \mathfrak{p}^{+}[B(F)]$ is bounded by Remark \ref{R:bounded}, but it is unknown that $\mathfrak{p}^{+}[A(G)] \#_{t} \mathfrak{p}^{+}[B(F)]$ is convex for $t \in [0,1]$. This is connected with the question in Remark \ref{R:convex}.

\item[(3)] In Theorem \ref{T:order_property} we have seen that $A(G) \leq B(G)$ for partial positive definite matrices $A(G)$ and $B(G)$ implies $\det (\widehat{A}) \leq \det (\widehat{B})$, where $\widehat{A}$ is the maximum determinant positive definite completion of $A(G)$. It naturally occurs that $\widehat{A} \leq \widehat{B}$. If it is true, then Theorem \ref{T:order_property} holds automatically.

\end{itemize}



\bibliographystyle{plain}
\bibliography{references.bib}

\end{document}